\documentclass[reqno,11pt]{amsart}

\usepackage{mathrsfs}
\usepackage{amscd}
\usepackage{amsmath}
\usepackage{latexsym}
\usepackage{amsfonts}
\usepackage{amssymb}
\usepackage{amsthm}
\usepackage{graphicx}
\usepackage{hyperref}
\usepackage{makecell}
\usepackage{array}
\usepackage{booktabs}
\usepackage{multirow}
\usepackage{color,xcolor}
\usepackage{cite}

\newtheorem{theorem}{Theorem}[section]

\newtheorem{corollary}[theorem]{Corollary}
\newtheorem{lemma}[theorem]{Lemma}

\newtheorem{remark}[theorem]{Remark}

\numberwithin{equation}{section}

\title[Direct and inverse problems]{Direct and inverse problems for the Schr\"{o}dinger operator in a three-dimensional planar waveguide}

 \author[P. Li]{Peijun Li}
\address{Department of Mathematics, Purdue University, West Lafayette, Indiana
47907, USA}
\email{lipeijun@math.purdue.edu}

  \author[X. Yao]{Xiaohua Yao}
\address{School of Mathematics and Statistics, Central China Normal University,
Wuhan 430079, China}
\email{yaoxiaohua@mail.ccnu.edu.cn}

  \author[Y. Zhao]{Yue Zhao}
\address{School of Mathematics and Statistics, Central China Normal University,
Wuhan 430079, China}
\email{zhaoyueccnu@163.com}

\thanks{The research of PL is supported in part by the NSF grant DMS-1912704. The research of XY is supported in part by NSFC (No. 11771165). The research of YZ is supported in part by NSFC (No. 12001222).}

\subjclass[2010]{35R30, 78A46}

\keywords{the Schr\"{o}dinger operator, waveguide, resolvent estimate, inverse source problem, stability}

\begin{document}

\begin{abstract}
In this paper, we study the meromorphic continuation of the resolvent for the Schr\"{o}dinger operator 
in a three-dimensional planar waveguide. We prove the existence of a resonance-free region and an upper bound for the resolvent. 
As an application, the direct source problem is shown to have a unique solution under some appropriate assumptions. 
Moreover, an increasing stability is achieved for the inverse source problem of the Schr\"{o}dinger 
operator in the waveguide by using limited aperture Dirichlet data only at multiple frequencies. 
The stability estimate consists of the Lipschitz type data discrepancy and the high frequency tail of the
source function, where the latter decreases as the upper bound of the frequency increases. 

\end{abstract}

\maketitle

\section{Introduction}

In this paper, we consider the direct and inverse problems for the Schr\"{o}dinger operator in a three-dimensional planar waveguide. Let $D\subset\mathbb R^3$ be an infinite slab between two parallel hyperplanes $\Gamma_1$ and $\Gamma_2$, as shown in Figure \ref{figure}.  Without loss of generality, we assume that
\[
D = \{x = (x^\prime, x_3)\in\mathbb R^3: x' = (x_1, x_2)\in\mathbb R^2, 0<x_3<L\}, \quad L>0,
\]
and
\[
\Gamma_1 = \{x\in\mathbb R^3: x_3 = L\}, \quad \Gamma_2 = \{x\in\mathbb R^3: x_3 = 0\}.
\]

Consider the following boundary value problem of the Schr\"odinger operator: 
\begin{equation}\label{main_eq}
\begin{split}
-\Delta u(x,\kappa)+ V(x) u(x,\kappa) - \kappa^2 u(x,\kappa) &=f(x),\quad x\in D, \\
u(x, \kappa) &= 0, \quad\quad x\in\Gamma_1\cup\Gamma_2,
\end{split}
\end{equation}
where $\kappa>0$ is the wavenumber, $V$ is the potential function, and $f$ is the source function. Let $\Omega\subset D$ be a convex bounded domain with smooth boundary such that $\partial\Omega = \Gamma\cup\gamma_1\cup\gamma_2$, where $\gamma_i, i=1, 2$ are non-empty closed domains of $\Gamma_i$ and $\Gamma = \partial\Omega\cap D$. 
Furthermore, we assume that $f\in L^2(D)$, $V(x)\in L^\infty(D)$  and they both have compact supports contained in $\overline{\Omega}$. The Sommerfeld radiation condition is imposed to ensure the well-posedness of the
problem:
\begin{align}\label{src}
\lim_{r\to\infty}r^{\frac{1}{2}} (\partial_r u-\mathrm{i}\kappa u)=0,\quad r=|x'|
\end{align}
uniformly in all directions $\hat{x'} = x'/|x'|$. 

In this work, we study both the direct and inverse problems of \eqref{main_eq}. Given the potential $V$ and the source $f$, the direct problem is to investigate the resolvent $(-\Delta + V - \lambda)^{-1}$ in $D$ for $\lambda\in\mathbb C$; the inverse problem is to determine the source $f$ from the given potential $V$ and the boundary measurement $u(x, \kappa)\vert_{\Gamma}$ corresponding to the wavenumber $\kappa$ in a finite interval.

\begin{figure}\label{figure}
\centering
\includegraphics[width=0.4\textwidth]{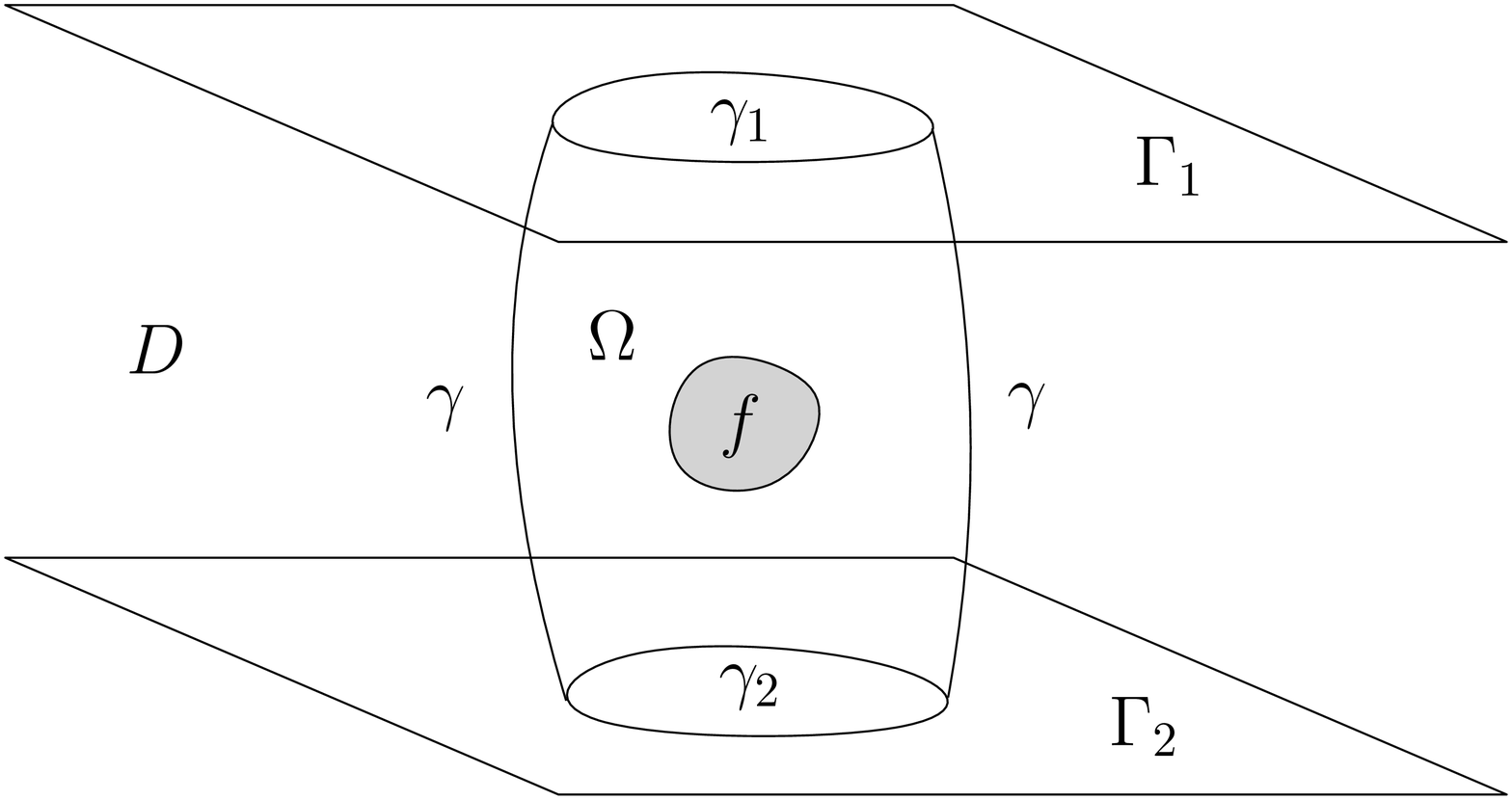}
\caption{The problem geometry. }
\end{figure}

There is extensive literature on the subject of scattering problems in a planar waveguide, see e.g., \cite{AGA, BGWX, collin, marcuse}. An important application of the waveguide problem is to provide a simple but effective model for the propagation of time-harmonic acoustic waves in oceans \cite{AGA, BGWX}, where the scattered field is assumed to satisfy the Neumann and  Dirichlet boundary condition, respectively. Our method can be directly applied to this model problem. We refer to \cite{KLU, LU} and the references cited therein for the corresponding inverse boundary value problems, where the data is the Dirichlet-to-Neumann (DtN) map and the main mathematical tool is the construction of complex geometric optics solutions. However, the analysis on the scattering resonances for the resolvent in this geometry remains open. There are few works on the inverse source problem and the available results are mainly focused on the numerics \cite{AAB, BIT, liu, LXZ}. 

For the direct problem, we first study the free resolvent in $\mathbb R^2$ and its meromorphic 
continuation. We obtain a resonance-free region and the resolvent estimates. Next is to consider the resolvent in a planar waveguide in $\mathbb R^3$. We start with the case of $V = 0$ and utilize the Fourier decomposition with respect to the variable $x_3$, which leads to a series of the scattered fields consisting of the Fourier coefficients. The Fourier coefficients are required to satisfy the two-dimensional Helmholtz equation and the Sommerfeld radiation condition \eqref{src}. Consequently, for any positive wave number, a finite number of the Fourier coefficients propagate to the infinity, while an infinite number of Fourier coefficients decay exponentially. The former correspond to the propagating wave modes and the latter correspond to the evanescent wave modes. Hence, there is an issue for the waveguide scattering in free space: resonances might occur at some special frequencies which give rise to nonuniqueness for the direct problem \cite{AGA}. On the other hand, a key question must be answered when studying the stability of the inverse source problem by using multi-frequency data: whether we can find a resonance-free region which contains an infinite interval of the positive real axis? Clearly, the resonances may also bring difficulties in this aspect. To resolve this issue, we assume that the source function only has a finite number of Fourier coefficients. Meanwhile, this assumption has another important role on solving the direct and inverse problems. When examining the resolvent in the lower complex half-plane, analysis of the series solution for the scattered field shows that it may not converge. The reason is that the analytic continuation of the complex wavenumbers corresponding to the evanescent modes from the upper to the lower complex half-plane will result in exponentially growing Fourier coefficients. The finite band assumption of the source function also helps us handle this issue. Next, for $V\neq 0$, using the perturbation arguments, we present the meromorphic continuation of the resolvent from the upper to the lower complex half-plane, and prove the existence of a resonance-free region and an upper bound for the resolvent, which gives the well-posedness of the direct problem. We believe that under the above assumption, the resonances occurring at special frequencies in a tubular waveguide (see e.g., (2.6) in \cite{BCM}) can also be  dealt with by restricting the complex wavenumber on the domain outside a sufficiently large disk in $\mathbb C$.  

In general, it is well known that there is no uniqueness for inverse source problems at a fixed frequency due to the existence of non-radiating sources \cite{bc-77,ds-82}. Computationally, a more serious issue is the lack of stability, i.e., a small variation of the data might lead to a huge error in the reconstruction. Hence it is crucial to study the stability of the inverse source problems. The first stability result was obtained in \cite{blt} for the inverse source problem of the Helmholtz equation by using multi-frequency data. Later on, the increasing stability was studied for the inverse source problems of various wave equations including the acoustic, elastic, and electromagnetic wave equations, as well as the wave equation with the biharmonic 
Schr\"{o}dinger operator \cite{blz, cheng2016increasing, ei-18, ei-19, li2017increasing, LZZ, LYZ}. A more recent study on the 
stability for the inverse medium problem can be found in \cite{BT}. In this work, we show
an increasing stability for the inverse source problem formulated in a three-dimensional planar waveguide by using 
multi-frequency limited aperture Dirichlet data $u(x, \kappa)\vert_{\Gamma}$. The method is expected to be applicable to the  waveguide problems of infinite unperturbed tubular structures and to electromagnetic waves. We refer to \cite{MS} and \cite{LZ} for the tubular waveguide problem and the inverse source problem of electromagnetic waves in an inhomogeneous medium, respectively. The analysis also sheds light on the study of scattering problems for the Schr\"{o}dinger operator with a periodic potential in the structure of $\mathbb R\times \mathbb T^2$ where $\mathbb T^2 = \mathbb R^2\slash\mathbb Z^2$ is the torus. A typical example of the scattering problem in $\mathbb R\times \mathbb T^2$ is the Maxwell equations in periodic structures. We refer to the monograph \cite{BL} for recent developments in mathematical analysis and computational methods for solving Maxwell's equations in periodic structures. 

Motivated by \cite{LYZ, LZZ}, we consider an eigenvalue problem for the Schr\"{o}dinger operator in $\Omega$ and deduce integral equations, which connect the scattering data $u(x, \kappa)\vert_{\Gamma}$ and the unknown source function $f$. It is required to obtain a bound of the analytic continuation of the data from the given data to the higher frequency data in order to show the stability. We make use of the resolvent of the Schr\"{o}dinger operator in the waveguide to obtain a resonance-free region of the data with respect to the complex wavenumber, which contains an infinite interval of the positive real axis. Furthermore, a corresponding upper bound of the resolvent is needed, which also comes from the analysis of the resolvent. In fact, the key idea for this work is to utilize the scattering theory to study the inverse scattering problem. As mentioned in \cite{CCH}, the theory of scattering resonances is a rich and beautiful part of scattering theory. It studies the meromorphic continuation of the resolvent. A comprehensive account on this subject can be found in the monograph \cite{Dyatlov}. The stability estimate derived in this paper consists of the Lipschitz type of data discrepancy and the high frequency tail of the source function. The latter decreases as the frequency of the data increases, which implies that the inverse problem is more stable when the higher frequency data is used. We also mention that only the Dirichlet data is required for the analysis.

This paper is organized as follows. In Section \ref{ds}, we discuss the direct problem. The free resolvent in $\mathbb R^2$ is discussed. The resolvent is introduced for the Schr\"{o}dinger operator in a planar waveguide in $\mathbb R^3$ and its holomorphic domain and upper bound are obtained, which lead to the well-posedness of the direct problem. Section \ref{ip} is devoted to the stability analysis of the inverse source problem by using multi-frequency limited aperture Dirichlet data. The paper is concluded with some general remarks and possible future work in Section \ref{con}.

\section{The direct problem}\label{ds}

In this section, we discuss the free resolvent in $\mathbb R^2$ and the resolvent for the planar waveguide in $\mathbb R^3$. As a result of the resolvent analysis, the direct problem is shown to have a unique solution. 

\subsection{The free resolvent in $\mathbb R^2$}\label{2}

Denote the free resolvent in $\mathbb R^2$ by $G (\lambda):= (-\Delta - \lambda^2)^{-1}, \lambda\in\mathbb C$, which has the kernel (cf. \cite{FY06})
\begin{align*}
G (\lambda, x, y) =  C e^{{\rm i} \lambda |x - y|} \int_0^\infty e^{-t} t^{-\frac{1}{2}} \Big( \frac{t}{2} - {\rm i}\lambda |x - y| \Big)^{-\frac{1}{2}} {\rm d}t,
\end{align*}
where $C$ is a positive constant. A simple calculation yields 
\begin{align*}
|G (\lambda, x, y)| \lesssim  \frac{e^{-\Im\lambda |x - y|}}{|\lambda|^{\frac{1}{2}}|x - y|^{\frac{1}{2}}} \int_0^\infty e^{-t} t^{-\frac{1}{2}} {\rm d}t \lesssim  
\frac{e^{-\Im\lambda|x - y|}}{|\lambda|^{\frac{1}{2}}|x - y|^{\frac{1}{2}}}.
\end{align*}
Hereafter, the notation $a\lesssim b$ stands for $a\leq Cb,$ where $C>0$ is a generic constant which may change step by step in the proofs.

Then it is easy to verify that when $\lambda\in\mathbb C^+=\{\lambda=\lambda_1+{\rm i}\lambda_2: \lambda_2>0\}$ the operator $G(\lambda)$ is well-defined on $L^2(\mathbb R^{2})$ via the explicit expression
\begin{align}\label{rg}
G (\lambda)(f) = \int_{\mathbb R^{2}} G (\lambda, x, y) f(y){\rm d}y,
\end{align}
and is analytic as a family of operators $G (\lambda): L^2(\mathbb R^{2})\rightarrow L^2(\mathbb R^{2})$.
Moreover, if we consider $G (\lambda)$ as an operator mapping $L^2_{\rm comp}(\mathbb R^{2})$ onto $L^2_{\rm loc}(\mathbb R^{2})$ in the sense that for any $\rho\in C_0^\infty(\mathbb R^{2})$,  the operator $\rho G (\lambda) \rho: L^2(\mathbb R^{2})\to L^2(\mathbb R^{2})$ is bounded, then the operator $G (\lambda)$ can be extended into an analytic family of operators for all $\lambda \in \mathbb{C}$. Specifically, we have the following result concerning a resonance-free region and an estimate for the resolvent $G (\lambda)$ in $\mathbb R^{2}$.

\begin{theorem}\label{free_estimate_2d}
The resolvent defined in \eqref{rg} continues analytically to an analytic family of operators for $\lambda\in\mathbb C^+\setminus\{0\}$
\begin{align*}
G (\lambda): L^2_{\rm comp}(\mathbb R^{2})\to L^2_{\rm loc}(\mathbb R^{2}),
\end{align*}
such that for each $\rho\in C_0^\infty(\mathbb R^{2})$ 
\begin{align}\label{free}
\|\rho G(\lambda) \rho\|_{L^2(\mathbb R^2)\rightarrow H^j(\mathbb R^2)}\lesssim |\lambda|^{-\frac{1}{2}} (1+|\lambda|^2)^{\frac{j}{2}} e^{L (\Im\lambda)_-}, \quad j=0, 1, 2,
\end{align}
where $t_{-}:=\max\{-t,0\}$ and $L>{\rm diam}({\rm supp}\rho): = \sup\{|x - y| : x, y \in {\rm supp}\rho\}$.
\end{theorem}

\begin{proof}
Let $\rho\in C_0^\infty(\mathbb R^{2})$ with $\rm supp\rho\subset\Omega$. For each $\lambda\neq 0$ we define the operator $\rho G (\lambda)\rho$ by
\begin{align}\label{kernel2}
(\rho G (\lambda)\rho f)(x) = \int_{\mathbb R^{2}}\rho(x)G (\lambda, x, y) \rho(y)f(y){\rm d}y.
\end{align}
It can be verified that $\rho G (\lambda)\rho$ is bounded on $L^2(\mathbb R^{2})$ and satisfies the estimate
\begin{align*}
\|\rho G (\lambda)\rho\|_{L^2(\mathbb R^{2})\rightarrow L^2(\mathbb R^{2})} &\lesssim |\lambda|^{-\frac{1}{2}} e^{L(\Im\lambda)_-} \times\\
&\qquad \Big(\int_{\mathbb R^{2}}\int_{\mathbb R^{2}} \rho^2(x)\frac{1}{|x-y|}\rho^2(y){\rm d}x{\rm d}y\Big)^{1/2}\\
&\lesssim |\lambda|^{-\frac{1}{2}} e^{L(\Im\lambda)_-},
\end{align*}
which implies that the operator $\rho G (\lambda) \rho$ belongs to the Hilbert--Schmidt class. Moreover, for any $f, g\in L^2(\mathbb R^{2})$, by the explicit expression \eqref{kernel2} of $\rho G (\lambda)\rho$, it is easy to prove that the function $I_0(\lambda):=\langle \rho G(\lambda)\rho f, g\rangle_{L^2(\mathbb R^{2})}$ is analytic in $\mathbb{C}\setminus\{0\}$. Consequently, $\rho G (\lambda)\rho$ is an analytic family of compact operators for $\lambda\in\mathbb{C}\setminus\{0\}$.

It suffices to show the case $j=2$ in order to prove \eqref{free}. By the standard elliptic estimate, one has for any $\Omega\subset\subset W \subset \mathbb{R}^{2}$ that 
\begin{align*}
\|u\|_{H^2(\Omega)}\leq C\big(\|u\|_{L^2(W)} + \|\Delta u\|_{L^2(W)}\big).
\end{align*}
Taking $\tilde{\rho}\in C^\infty_0(\Omega)$ such that $\tilde{\rho} = 1$ near the support of $\rho$, we obtain 
\[
\|\rho u\|_{H^2(\Omega)}\leq C\big(\|\tilde{\rho}u\|_{L^2(\mathbb{R}^{2})} + \|\tilde{\rho}\Delta u\|_{L^2(\mathbb{R}^{2})}\big).
\]
Thus letting $u = G (\lambda)(\rho f), f\in L^2(\mathbb R^n)$ gives 
\[
\|\rho G (\lambda)(\rho f)\|_{H^2(\Omega)}\leq C\big(\|\tilde{\rho}G (\lambda)(\rho f)\|_{L^2(\mathbb R^{2})} + \|\tilde{\rho}\Delta(G (\lambda)(\rho f))\|_{L^2(\mathbb R^{2})}\big).
\]
Noting 
\begin{align*}
\|\tilde{\rho}\Delta(G (\lambda)(\rho f))\|_ {L^2(\mathbb R^2)}&=\| \rho f + \tilde{\rho}\lambda^2G (\lambda)(\rho f)\|_{L^2(\mathbb R^2)}\\& \lesssim 
|\lambda|^{-\frac{1}{2}}(1+
|\lambda|^2) \big(e^{L(\Im\lambda)_-} + e^{L(\Re\lambda)_-}\big)\|f\|_{L^2(\Omega)},
\end{align*}
we obtain 
\[
\|\rho G (\lambda)\rho\|_{L^2(\Omega)\rightarrow H^2(\Omega)}\lesssim 
|\lambda|^{-\frac{1}{2}}(1+ |\lambda|^2)
\big(e^{L(\Im\lambda)_-} + e^{L(\Re\lambda)_-}\big).
\]
Finally, the cases for $j = 1$ follow by the interpolation between $j=0$ and $j=2$.
\end{proof}

\subsection{The resolvent for the waveguide in $\mathbb R^3$}

Denote the resolvent in the waveguide by $R_V(\lambda) = (-\Delta + V - \lambda^2)^{-1}$, which reduces to $R_0(\lambda) = (-\Delta - \lambda^2)^{-1}$  when $V = 0$. For $\lambda>0$, taking the Fourier decomposition with respect to the variable $x_3\in [0, L]$, we have
\begin{align*}
u(x^\prime, x_3) = \sum_{n = 1}^\infty u_n(x^\prime) \sin(\alpha_n x_3),\quad 
f(x^\prime, x_3) = \sum_{n = 1}^\infty f_n(x^\prime) \sin(\alpha_n x_3),
\end{align*}
and
\begin{align*}
\big(-\Delta_{x^\prime} - (\lambda^2 - \alpha_n^2)\big)u_n(x^\prime) = f_n(x^\prime), \quad x^\prime \in \mathbb R^2, \, n \in\mathbb N,
\end{align*}
where $\alpha_n=n\pi/L$, the Fourier coefficients $u_n$ of $u$ and $f_n$ of $f$ are given by
\begin{align}\label{fourier}
u_n(x^\prime) = \frac{2}{L} \int_0^L u(x) \sin(\alpha_n x_3) {\rm d}x_3,\quad 
f_n(x^\prime) = \frac{2}{L} \int_0^L f(x) \sin(\alpha_n x_3) {\rm d}x_3.
\end{align}
For $f\in L^2(D)$ we have the Parseval identity
\begin{align}\label{pif}
\|f\|^2_{L^2(D)} = \frac{L}{2} \sum_{n = 1}^\infty \|f_n\|^2_{L^2(\mathbb R^2)}.
\end{align}

It follows from the Sommerfeld radiation condition \eqref{src} that we have the following explicit representation for 
$u_n(x^\prime)$: 
\[
u_n(x^\prime) = (-\Delta_{x^\prime} - \beta^2_n)^{-1} f_n = G(\beta_n(\lambda))(f_n),
\]  
where $G$ is the free resolvent in $\mathbb R^2$ introduced in Section \ref{2} and $\beta^2_n(\lambda) = \lambda^2 - \alpha_n^2$ with $\Im\beta_n>0$. In the case when $n\in\mathbb N$ is such that $\lambda = \alpha_n$, 
the resonance phenomenon may occur at those special frequencies which may cause nonuniqueness of the solution for the waveguide problem in $\mathbb R^3$. Therefore,  we assume that $\lambda\neq \alpha_n$ for all $n\in\mathbb N$ to exclude the resonance. For a more detailed discussion on this issue, we refer to \cite{AGA} and (II. iii) and assumption (A. I) in Appendix A of \cite{KLU}. Thus, we have a formal representation for the free resolvent 
\begin{align}\label{sum1}
R_0(\lambda)(f) = \sum_{n = 1}^\infty  (-\Delta_{x^\prime} - \beta^2_n(\lambda))^{-1} (f_n) \sin(\alpha_n x_3), \quad \lambda\in{\mathbb R}^+\backslash\cup_{n=0}^N\{\alpha_n\}. 
\end{align}

Next, choosing $\Im\beta_n(\lambda)>0$ for $\lambda\in\mathbb C^+$, we extend the domain of $R_0(\lambda)$ from the positive real axis to the upper complex half-plane $\mathbb C^+$. Denote
\[
\beta_j^2(\lambda) = \gamma_n (\lambda) + {\rm i} \eta_n (\lambda), \quad \lambda\in\mathbb C^+,
\]
where
\[
\gamma_n (\lambda) = \lambda_1^2 - \lambda_2^2 - \alpha_n^2, \quad \eta_n (\lambda) = 2\lambda_1\lambda_2.
\]
A simple calculation yields
\[
\beta_n (\lambda) = a_n (\lambda) + {\rm i}b_n (\lambda),
\]
where
\begin{align}\label{ribeta}
a_n (\lambda) = {\Re}\beta_n = \bigg(\frac{\sqrt{\gamma_n^2 + \eta_n^2} + \gamma_n}{2}\bigg)^{1/2}, \quad 
b_n (\lambda) = {\Im}\beta_n = \bigg(\frac{\sqrt{\gamma_n^2 + \eta_n^2} - \gamma_n}{2}\bigg)^{1/2}.
\end{align}
Hence, for each $\lambda\in\mathbb C^+,$ it is easy to see that ${\rm Im}\beta_n \sim \alpha_n$ uniformly for all $z \in B_\delta(\lambda)$, where 
$B_\delta(\lambda) := \{z: |z - \lambda|<\delta\}\subset \mathbb C^+$ is a neighborhood of $\lambda$. Then we have 
\[
e^{{\rm i}\beta_n(z)|x^\prime - y^\prime|} \sim e^{-\alpha_n|x^\prime - y^\prime|},
\]
which implies that the series in \eqref{sum1} converges absolutely. Hence $R_0(\lambda)$ is analytic for all $\lambda$ in the upper complex half-plane $\mathbb C^+$ and is continuous up to the boundary $\mathbb R\backslash\cup_{n=0}^N\{\alpha_n\}$.

In what follows, we study the continuation of the resolvent $R_0(\lambda)$ to the lower complex half-plane 
$\mathbb C^- := \{\lambda=\lambda_1+{\rm i}\lambda_2\in\mathbb C: \lambda_2<0\}$. It is required to have an extension of the domain for $a_n$ and $b_n$ in \eqref{ribeta} from $\mathbb C^+$ to $\mathbb C^-$. Let $a_n = a_n(\lambda_1, \lambda_2)$ and $b_n=b_n(\lambda_1, \lambda_2)$ be functions of the variables $\lambda_1$ and $\lambda_2$. Noting that $\partial_{\lambda_2} a_n(\lambda) = 0$ and $b_n(\lambda) = 0$ on the interface $\{\lambda\in\mathbb C: \lambda_2 = 0\}$, we apply the even extension to $a_n(\lambda_1, \lambda_2)$ and the odd extension to $b_n(\lambda_1, \lambda_2)$ with respect to the variable $\lambda_2$, which gives the extension of $\beta_n(\lambda)$ to the whole complex plane as 
follows 
\begin{align}\label{real}
\Re\beta_n = 
\begin{cases}
a_n(\lambda_1, \lambda_2), \quad &\lambda_2\geq 0,\\[7pt]
a_n(\lambda_1, -\lambda_2), \quad &\lambda_2< 0,
\end{cases}
\end{align}
and
\begin{align}\label{imaginary}
\Im\beta_n = 
\begin{cases}
b_n(\lambda_1, \lambda_2), \quad &\lambda_2\geq 0,\\[7pt]
-b_n(\lambda_1, -\lambda_2), \quad &\lambda_2< 0.
\end{cases}
\end{align}

It is clear to note that $\beta_n(\lambda) = \Re\beta_n(\lambda) + {\rm i} \Im\beta_n(\lambda)$ is analytic in the whole complex plane $\mathbb C$. However, the imaginary part $\Im\beta_n(\lambda)$ is negative in $\mathbb C^-$ now. Therefore, it holds that 
\[
e^{{\rm i}\beta_n(z)|x^\prime - y^\prime|} \sim e^{\alpha_n|x^\prime - y^\prime|},
\]
which means that the series in \eqref{sum1} may not converge. To resolve this issue and study the meromorphic continuation of the resolvent $R_V(\lambda)$, we make the following two assumptions:

\begin{itemize}

\item[(A)] The source function $f(x)$ has only a finite number of Fourier coefficients defined in \eqref{fourier};

\item[(B)] The potential function $V(x)$ only depends on the $x^\prime$ variable.

\end{itemize}
By Assumption (A), there exits a positive integer $N>0$ such that the source function $f$ can be written as 
\begin{equation*}
f(x) = \sum_{n = 1}^N f_n(x^\prime)\sin(\alpha_n x_3). 
\end{equation*}

Denote
\[
L_N^2(D) = \{f\in L^2(D): f\,\, \text{satisfies Assumption (A) with $N$ Fourier coefficients}\},
\] 
which is a closed subspace of $L^2(D)$ equipped with the usual $L^2$-norm, and
\[
H_N^j(D) = \{u \in H^j(D): u\,\, \text{satisfies Assumption (A) with $N$ Fourier coefficients}\},
\] 
which is a closed subspace of $H^j(D)$ equipped with the usual $H^j$-norm for each $j\in\mathbb N$. 
Denote by $L_{\rm comp}^2(D)$ the space of functions in $L^2(D)$ with compact supports and $C_0^\infty(D)$ the space of smooth functions with compact supports contained in $\overline{D}$. Denote by $C_{0, \,x^\prime}^\infty(D)$ the subspace of  $C_0^\infty(D)$ consisting of smooth functions which have compact supports and depend only on the $x^\prime$ variable, i.e., 
\[
C_{0, \,x^\prime}^\infty(D) = \{\rho(x): \rho(x) = \eta(x^\prime), \,\eta(x^\prime)\in C_0^\infty(\mathbb R^2)\}.
\]

\begin{theorem}\label{free_estimate}
Under Assumptions (A) and (B), the resolvent operator $R_0(\lambda)$ for $\lambda\in\mathbb C^+$ can be extended into an analytic family of operators for all $\lambda\in\mathbb{C}\backslash\cup_{n=0}^N\{\alpha_n\}$ as
\begin{align*}
R_0(\lambda): L^2_{N, \rm comp}(\Omega)\to L^2_{N, \rm loc}(\Omega),
\end{align*}
such that for each $\rho\in C_{0, \,x^\prime}^\infty(\Omega)$ with ${\rm supp}(\rho)\subset \overline{\Omega}$ 
\begin{align}\label{free2}
\|\rho R_0(\lambda) \rho\|_{L_N^2(\Omega)\rightarrow H_N^j(\Omega)}\lesssim |\lambda|^{-1/2} (1 + |\lambda|^2)^{j/2}\max_{1\leq n\leq N} e^{T(\Im\beta_n)_-}\quad j=0, 1, 2,
\end{align}
where $T > \sup \{|x - y|: x, y\in{\rm supp}\rho\}$. 
\end{theorem}

\begin{proof}

We first extend $\beta_n(\lambda)$ in \eqref{sum1} to the whole complex plane by \eqref{real} and \eqref{imaginary}, then $\beta_n(\lambda)$ is an analytic function of $\lambda$. It follows from Theorem \ref{free_estimate_2d} that for any $\eta(x^\prime)\in C_0^\infty(\mathbb R^2)$, the resolvent $\eta G (\beta_n(\lambda)) \eta : L^2(\mathbb R^2)\rightarrow H^j(\mathbb R^2)$ is analytic in $\mathbb C\backslash\{0\}$ excluding the resonance $\lambda = \alpha_n$
and satisfies the estimates
\begin{align}\label{est_beta}
\|\eta G (\beta_n(\lambda)) \eta\|_{L^2(\mathbb R^2)\rightarrow H^j(\mathbb R^2)} \lesssim |\lambda|^{-1/2}(1 + |\lambda|^2)^{j/2} e^{L(\Im\beta_n(\lambda))_-}, \quad j = 0, 1, 2,
\end{align} 
where $L > \sup \{|x - y|: x, y\in{\rm supp}\eta\}$.

Let $\rho\in C_{0, \, x^\prime}^\infty(\Omega)$. Then if $f(x)\in L_{N}^2(D)$, $(\rho f)(x)$ is also 
in $L^2_N(D)$. Define the operator $\rho R_0(\lambda)\rho$ by
\begin{align*}
(\rho R_0(\lambda) \rho f )(x) = \sum_{n = 1}^N \rho(x) G (\beta_n(\lambda))(\rho f)_n(x^\prime) \sin(\alpha_n x_3) .
\end{align*}
Therefore, for $\lambda\in\mathbb{C}\backslash\cup_{n=0}^N\{\alpha_n\}$, we have from \eqref{est_beta} and the Parseval identity \eqref{pif} that 
\begin{align*}
\|\rho R_{0}(\lambda) \rho f \|_{H_N^j(D)}&\lesssim \Sigma_{n =1}^N |\lambda|^{-1/2}(1 + |\lambda|^2)^{j/2} e^{T(\Im\beta_n(\lambda))_-} \|f_n\|_{L^2(\mathbb R^2)}\\
&\lesssim  |\lambda|^{-1/2}(1 + |\lambda|^2)^{j/2} \max_{1\leq n\leq N} e^{T(\Im\beta_n(\lambda))_-}  \Sigma_{n =1}^N \|f_n\|_{L^2(\mathbb R^2)}\\
&\lesssim  |\lambda|^{-1/2}(1 + |\lambda|^2)^{j/2} \max_{1\leq n\leq N} e^{T(\Im\beta_n(\lambda))_-} \|f\|_{L_N^2(D)},  \quad j = 0, 1, 2,
\end{align*}
where $T > \sup \{|x - y|: x, y\in{\rm supp}\rho\}$. Consequently, we obtain that 
\[
\rho R_{0}(\lambda) \rho: L_N^2(D)\rightarrow H_N^j(D), \quad j = 0, 1, 2
\]
is an analytic family of operators for $\lambda\in\mathbb{C}\backslash\cup_{n=0}^N\{\alpha_n\}$ and 
\begin{align*}
\|\rho R_{0}(\lambda) \rho\|_{L_N^2(D)\rightarrow H_N^j(D)} \lesssim |\lambda|^{-1/2}(1 + |\lambda|^2)^{j/2} \max_{1\leq n\leq N} e^{T(\Im\beta_n)_-},
\quad j = 0, 1, 2,
\end{align*}
which completes the proof. 
\end{proof}

By Theorem \ref{free_estimate}, the scattering problem \eqref{main_eq}--\eqref{src} has a unique solution for all the positive wavenumbers excluding the resonances when $V(x)\equiv 0$, which is stated in the following result.

\begin{corollary}\label{wp1}
Let $V(x)\equiv 0$. For all $\kappa\in\mathbb{R}^+\backslash\cup_{n=1}^N\{\alpha_n\}$, the scattering problem \eqref{main_eq}--\eqref{src} admits a unique solution $u\in H^2(\Omega)$ such that 
\[
\|u\|_{H^2(\Omega)}\lesssim \|f\|_{L^2(\Omega)}.
\]
\end{corollary}

Denote by $T: L^2_{N, \rm comp}(D)\rightarrow L^2_{N, \rm loc}(D)$ an operator such that for any $\rho\in C_{0, \, x^\prime}^\infty(D)$ the operator $\rho T \rho: L_N^2(D)\rightarrow L_N^2(D)$ is bounded. Below is the analytic Fredholm theory. The result is classical and the proof may be found in many references, e.g., \cite[Theorem 8.26]{CK}.

\begin{theorem}
Let $D$ be a domain in $\mathbb C$ and let $K: D \rightarrow \mathcal{L}(X)$ be an operator
valued analytic function such that $K(z)$ is compact for each $z\in D$. Then either
\begin{itemize}

\item[(a)] $(I - K(z))^{-1}$ does not exist for any $z\in D$ or

\item[(b)] $(I - K(z))^{-1}$ exists for all $z\in D\backslash S$ where $S$ is a discrete subset of $D$.

\end{itemize}
Here $X$ is a Banach space and $\mathcal{L}(X)$ denotes the Banach space
of bounded linear operators mapping the Banach space $X$ into itself.
\end{theorem}

The following theorem gives a meromorphic continuation of the resolvent $R_V(\lambda)$ from $\mathbb C^+$ to $\mathbb C^-$.
To apply the perturbation arguments, we assume that $V$ satisfies Assumption (B) such that the operator 
\[
T: u \rightarrow Vu
\]
maps the closed subspace $L_N^2(D)$ in $L^2(D)$ to $L_N^2(D)$ itself. In this way the application of the analytic Fredholm theory
stated above becomes available in the proof of the following theorem.
 
\begin{theorem}\label{meromorphic}
Under Assumptions (A) and (B), the resolvent
\[
R_V= (-\Delta + V - \lambda^2)^{-1}: L_N^2(D)\rightarrow L_N^2(D)
\]
is a meromorphic family of operators on the upper complex half-plane $\mathbb C^+$. Moreover, the family of operators $R_V(\lambda)$ can be extended into a meromorphic family of the whole complex plane $\mathbb C$ in the sense that 
$
\rho R_V(\lambda) \rho : L_N^2(D) \to H_N^2(D)
$
 is bounded for any $\rho\in C_{0, \, x^\prime}^\infty(D)$.
\end{theorem}

\begin{proof}
First we consider the case  $R_V(\lambda)$ for $\Im\lambda\gg 1$. Since $V$ satisfies Assumption (B), we have that 
$Vu\in L^2_N(D)$ for any $u\in L^2_N(D)$. Noting the equality 
\begin{align}\label{equality}
(-\Delta + V(x) -\lambda^2)R_0(\lambda) = (-\Delta - \lambda^2) R_0(\lambda) + V(x)R_0(\lambda) = I + V(x)R_0(\lambda)
\end{align}
and recalling the estimate \eqref{free} when $j = 0$ and $\Im\lambda>0$:
\[
\|R_0(\lambda)\|_{L_N^2(D)\rightarrow L_N^2(D)}\leq \frac{1}{\sqrt{|\lambda|}}, \quad \lambda \in \mathbb C^+,
\]
we obtain for $\Im\lambda\gg 1$ that
\begin{align*}
\|VR_0(\lambda)\|_{L_N^2(D)\rightarrow L_N^2(D)} \leq \|V\|_{L^\infty(\mathbb R^3)}\|R_0(\lambda)\|_{L_N^2(D)\rightarrow L_N^2(D)}\leq\frac{\|V\|_{L^\infty(D)}}
{\sqrt{|\lambda|}}\leq\frac{1}{2}.
\end{align*}
Hence the operator $I + VR_0(\lambda)$ is invertible and the Neumann series reads
\[
(I + VR_0(\lambda))^{-1} = \sum_{k=0}^\infty (-1)^k (VR_0(\lambda))^k.
\]
Combining with \eqref{equality} gives that
\[
R_V(\lambda) = R_0(\lambda) (I + VR_0(\lambda))^{-1}
\]
are well-defined bounded operators of $\mathcal B(L^2_N(D), L_N^2(D))$ for $\Im\lambda\gg 1$. Moreover, it is easy to see that $VR_0(\lambda)$ is an analytic family of compact operators on $\mathbb C^+$. Consequently, it follows from the analytic Fredholm theorem that $ (I + VR_0(\lambda))^{-1}$ is in fact a meromorphic family of operators for $\lambda\in\mathbb C^+$, which implies that $R_V(\lambda): L_N^2(D)\rightarrow L_N^2(D)$ is a meromorphic family of operators in $\mathbb C^+$.

Next, we consider the extension of $R_V(\lambda)$ from $\mathbb C^+$ to the whole complex plane $\mathbb C$ as the operator $L_{N, \rm comp}^2(D)\rightarrow  H^2_{N, \rm loc}(D)$. To this end, we define the following meromorphic family of operators:
\[
T(\lambda) = VR_0(\lambda): L_{N, \rm comp}^2(D)\rightarrow L^2_{N, \rm comp}(D).
\]
Since $V\in L^\infty_{\rm comp}(D)$ with a compact support, we can choose $\rho\in C_{0, \, x^\prime}^\infty(D)$ such that $\rho(x) = 1$ on $\text{supp} V$. Thus, by checking $\rho T(\lambda) = \rho V R_0(\lambda) = VR_0(\lambda) = T(\lambda)$, we know that $(1 - \rho)T(\lambda) = 0$, 
\[
(I + T(\lambda)(1 - \rho))^{-1} = I - T(\lambda)(1 - \rho)
\]
and
\[
(I + T(\lambda))^{-1} = (I + T(\lambda)\rho)^{-1} (I - T(\lambda)(1 - \rho)).
\]
Therefore
\begin{align}\label{expression1}
R_V(\lambda) = R_0(\lambda)(I + T(\lambda))^{-1} =  R_0(\lambda) (I + T(\lambda)\rho)^{-1} (I - T(\lambda)(1 - \rho)).
\end{align}
Note that
\[
I - T(\lambda)(1 - \rho): L^2_{N, \rm comp}(D)\rightarrow L^2_{N, \rm comp}(D)
\]
and
\[
R_0(\lambda):  L^2_{N, \rm comp}(D)\rightarrow H^2_{N, \rm loc}(D)
\]
are both meromorphic for $\lambda \in \mathbb{C}$. Hence in order to obtain the meromorphic extension of $R_V(\lambda)$ to $\mathbb C$, it suffices to prove
\[
(I + T(\lambda)\rho)^{-1} : L^2_{N, \rm comp}(D)\rightarrow L^2_{N, \rm comp}(D)
\]
is a meromorphic family of operators on $\mathbb C$. Since by $V(x) = V(x)\rho(x)$ we have $T(\lambda)\rho = V\rho R_0(\lambda)\rho$ and 
\begin{align*}
\|T(\lambda)\rho\|_{L^2_{N, \rm comp}(D)\rightarrow L^2_{N, \rm comp}(D)} &= \|V\rho R_0(\lambda)\rho\|_{L^2_{N, \rm comp}(D)\rightarrow L^2_{N, \rm comp}(D)}\\
&\leq\|V\|_{L^\infty} \|\rho R_0(\lambda)\rho\|_{L^2_{N, \rm comp}(D)\rightarrow L^2_{N, \rm comp}(D)}\\
&\leq C|\lambda|^{-1/2}\leq\frac{1}{2}
\end{align*}
for  $\Im\lambda\gg 1$. Hence it follows from the Neumann series the operator $(I + T(\lambda)\rho)^{-1}: L_N^2(D)\rightarrow L_N^2(D)$ exists for $\Im\lambda\gg 1$. Moreover, for any $\lambda\in\mathbb{C}$ the operator $T(\lambda)\rho = V\rho R_0(\lambda)\rho$ is compact on $L_N^2(D)$ by \eqref{free2}. Therefore, it follows from the analytic Fredholm theorem that $(I + T(\lambda)\rho)^{-1}: L_N^2(D)\rightarrow L_N^2(D)$ is meromorphic on $\mathbb C$.

Finally, it remains to show that $(I + T(\lambda)\rho)^{-1}$ is 
 $L^2_{N, \rm comp}(D)\rightarrow L^2_{N, \rm comp}(D)$. In fact, we can choose $\chi, \tilde{\chi}\in C_{0, \, x^\prime}^\infty(D)$ such that $\chi\rho = \rho$ and $\tilde{\chi}\chi = \chi$, then $(1 - \tilde{\chi})\rho = 0$. Moreover, when $\Im\lambda\gg 1$, by the Neumann series argument and $V\rho = V$, we have
\begin{align}\label{cutoff}
(1 - \tilde{\chi})(I + T(\lambda)\rho)^{-1}\chi &= (1 - \tilde{\chi})\chi + \sum_{k=1}^\infty (-1)^k (1 - \tilde{\chi}) (T(\lambda)\rho)^k\chi\notag\\
&= \sum_{k=1}^\infty (-1)^k (1 - \tilde{\chi}) (V\rho R_0(\lambda)\rho)^k\chi\notag\\
&= \sum_{k=1}^\infty (-1)^k (1 - \tilde{\chi}) (V\rho R_0(\lambda)\rho)(V\rho R_0(\lambda)\rho)^{k-1}\chi\notag\\
&=0,
\end{align}
where the last equality uses $(1 - \tilde{\chi})\rho = 0$. By the analytic continuation, \eqref{cutoff} remains true for all $\lambda\in\mathbb{C}\backslash\cup_{n=0}^N\{\alpha_n\}$. Therefore, by the expression \eqref{expression1} of $R_V(\lambda)$ we obtain that $R_V(\lambda)$ is meromorphic of $\lambda$ on $ \mathbb{C}$ as a family of operators $ L^2_{N, \rm comp}(D)\rightarrow H^2_{N, \rm loc}(D)$, which completes the proof.
\end{proof}

In the following theorem, we further give a resonance-free region and a resolvent estimate of $\rho R_V(\lambda)\rho: L_N^2(D)\rightarrow H_N^2(D)$ for a given $\rho\in C_{0, \, x^\prime}^\infty(D)$, which plays a crucial role in the stability analysis for the inverse problem. 

\begin{theorem}\label{bound_2}
Under Assumptions (A) and (B), let $V(x)\in L^\infty_{\rm comp}(D, \mathbb C)$. Then for any given $\rho\in C_{0, \, x^\prime}^\infty(\Omega)$ satisfying $\rho V = V$, i.e., $\text{supp} (V)\subset \text{supp} (\rho)\subset\overline{\Omega}$, there exists a positive constant $C$ depending on $\rho$ and $V$ such that
\begin{align}\label{bound_3}
\|\rho R_V(\lambda)\rho\|_{L_N^2(\Omega)\rightarrow H^j_N(\Omega)}\leq C|\lambda|^{-1/2}(1 + |\lambda|^2)^{j/2}
e^{2T(\Im\lambda)_-},\quad j = 0, 1, 2, 
\end{align}
where $\lambda\in \Omega_M$. Here $\Omega_M$ denotes the resonance-free region defined as
\begin{align*}
\Omega_M:=\{\lambda: {\Im}\lambda\geq -M {\rm log}|\lambda|, |\lambda|\geq C_0\},
\end{align*}
where $C_0$ is a positive constant and $M$ satisfies $0<M<\frac{1}{8T}$. 
\end{theorem}

\begin{proof}
Firstly, we notice that for $\lambda\in\Omega_M$ there holds 
\[
\max_{1\leq n\leq N}  e^{T(\Im\beta_n(\lambda))_-} \leq e^{2T(\Im\lambda)_-}.
\]
Then for $\lambda\in \Omega_M$, choosing $\rho_1\in C_{0, \, x^\prime}^\infty(D)$  such that $\rho_1V = V$,  we have
\begin{align*}
\|V R_0(\lambda)\rho_1\|_{L^2_N(D)\rightarrow L^2_N(D)} &= \|V\rho_1 R_0(\lambda)\rho_1\|_{L_N^2(D)\rightarrow L_N^2(D)}\\
&\lesssim \|V\|_{L^\infty(D)}|\lambda|^{-1/2} e^{2T(\Im\lambda)_-}\\
&\lesssim \|V\|_{L^\infty(D)}|\lambda|^{-1/2}e^{2TM\text{log}|\lambda|}\\
&\lesssim  \|V\|_{L^\infty(D)}|\lambda|^{-1/4}\leq\frac{1}{2},
\end{align*}
where we let $C_0\gg 1$ and $M<\frac{1}{8T}.$ To exclude the resonances, we assume $C_0>\alpha_N=N\pi/L$ at the same time.
Hence by the Neumann series argument we can prove that the inverse operator $(I + VR_0(\lambda)\rho_1)^{-1}$ exists for all $\lambda\in\Omega_M$, and
\begin{align}\label{free4}
\|(I + VR_0(\lambda)\rho_1)^{-1}\|_{L_N^2(D)\rightarrow L_N^2(D)} = \|(I + V\rho_1R_0(\lambda)\rho_1)^{-1}\|_{L_N^2(D)\rightarrow L_N^2(D)}\leq 2.
\end{align}
Now let $\rho\in C_{0, \, x^\prime}^\infty(D)$ such that $\rho V = V$. We have 
\[
\rho R_V(\lambda) \rho = \rho R_0(\lambda) \rho (I + VR_0(\lambda)\rho_1)^{-1} (I - VR_0(\lambda)(1 -\rho_1))\rho.
\]
Furthermore, by the estimate \eqref{free2} we obtain 
\begin{align}\label{free3}
\|\rho_1 R_0(\lambda)\rho_1\|_{L_N^2(D)\to L_N^2(D)} \leq C |\lambda|^{-1/2} e^{2T(\Im\lambda)_-}.
\end{align}
Combining  \eqref{free3} and \eqref{free4} gives the desired estimate for $j = 0$ as
\[
\|\rho R_V(\lambda)\rho\|_{L_N^2(D)\rightarrow L_N^2(D)} \leq C |\lambda|^{-1/2} e^{2T(\Im\lambda)_-}.
\]
 For the case $j = 2$, let $\tilde{\rho}\in C_{0, \, x^\prime}^\infty(D)$ such that $\tilde{\rho} = 1$ on $\text{supp}\rho$ and $\text{supp}\tilde{\rho} \subset \overline{\Omega}$. It holds that
\begin{align*}
\|\rho R_V(\rho f)\|_{H_N^2(\Omega)} &\lesssim  \|\tilde{\rho} R_V(\lambda)(\rho f)\|_{L_N^2(\Omega)} + \|\tilde{\rho} \Delta R_V(\lambda)
(\rho f)\|_{L_N^2(\Omega)}\\
&\lesssim \|\tilde{\rho} R_V(\lambda)(\rho f)\|_{L_N^2(\Omega)} + \|\tilde{\rho} (-\Delta + V) R_V(\lambda)(\rho f)\|_{L_N^2(\Omega)} + \|\tilde{\rho} V R_V(\lambda)(\rho f)\|_{L_N^2(\Omega)}\\
&\lesssim \|\tilde{\rho} R_V(\lambda)(\rho f)\|_{L_N^2(\Omega)} + |\lambda|^2 \|\tilde{\rho} R_V(\lambda)(\rho f)\|_{L_N^2(\Omega)} 
+ \|\rho f\|_{L_N^2(\Omega)}\\ 
&\quad + \|\tilde{\rho} V R_V(\lambda)(\rho f)\|_{L_N^2(\Omega)}\\
&\lesssim  (1+ |\lambda|)^2  \|\tilde{\rho} R_V(\lambda)(\rho f)\|_{L_N^2(\Omega)} \\
&\lesssim (1+ |\lambda|)^2 |\lambda|^{-1/2} e^{2T(\Im\lambda)_-}\|f\|_{L_N^2(\Omega)}
\end{align*}
for $\lambda\in\Omega_M$. Finally, the cases of $j =1$ follows by an application of the interpolation between $j = 0$ and $j = 2$. 
\end{proof}

Based on Theorem \ref{bound_2}, the scattering problem \eqref{main_eq}--\eqref{src} has a unique solution for all positive wavenumbers $\kappa\geq C_0$ and complex-valued $V(x)\in L^\infty_{\rm comp}(\Omega; \mathbb C)$, which is stated below. 

\begin{corollary}
Under Assumptions (A)-(B) and let $f(x)\in L_{\rm comp}^2(\Omega), V(x)\in L^\infty_{\rm comp}(\Omega; \mathbb C)$. Then for all positive wavenumbers $\kappa\geq C_0$ where $C_0$ is specified in Theorem \ref{bound_2}, the scattering problem \eqref{main_eq}--\eqref{src} admits a unique solution $u\in H^2(\Omega)$ which satisfies 
\[
\|u\|_{H^2(\Omega)}\lesssim \|f\|_{L^2(\Omega)}.
\]
\end{corollary}

\begin{remark}
Since resonances might occur at special frequencies $\alpha_n, n = 1, 2, ... , \infty$ for the scattering in a planar waveguide $D$ in $\mathbb R^3$  (see e.g., \cite{AGA}), Assumption (A) is indispensable in order to obtain a resonance-free region which contains an infinite interval of the positive real axis $\mathbb R^+$. 
\end{remark}
 
\begin{remark}
Using \cite[Theorem 2.1]{Dyatlov} for the free resolvent in $\mathbb R$ and following the above arguments we can prove similar results for the waveguide in $\mathbb R^2$. 
\end{remark}

\section{The inverse problem}\label{ip}

In this section, we study the uniqueness and stability of the inverse source problem under Assumptions (A) and (B). 
We also assume that $V\geq 0$ is a real-valued function.

We begin with the spectrum of the Schr\"odinger operator $-\Delta + V$ with the Dirichlet boundary condition. Let $\{\mu_j, \phi_j\}_{j=1}^\infty$ be the positive increasing eigenvalues and eigenfunctions of $-\Delta + V$ in $\Omega$, where $\phi_j$ and $\mu_j$
satisfy
\[
\begin{cases}
(-\Delta + V(x)) \phi_j(x)=\mu_j\phi_j(x)&\quad\text{in }\Omega,\\
 \phi_j(x)=0&\quad\text{on }\partial \Omega.
\end{cases}
\]
Again we require that $\mu_j \notin \{\alpha_n\}_{n = 1}^N, j = 1, ... , \infty$ to avoid the resonances. In fact, due to the finiteness of the resonances $\{\alpha_n\}_{n = 1}^N$, this requirement can be fulfilled by adjusting the scale of the domain $\Omega$.

Assume that $\phi_j$ is normalized such that
\[
\int_{\Omega}|\phi_j(x)|^2{\rm d}x=1.
\] 
Consequently, we obtain the spectral decomposition of $f$:
\[
f(x)=\sum_{j=1}^\infty f_j\phi_j(x),
\]
where
\begin{align}\label{integral_equation}
f_j=\int_{\Omega}f(x)\bar{\phi}_j(x){\rm
d}x.
\end{align}
It is clear that 
\[
 \|f\|^2_{L^2(\Omega)} = \sum_{j=1}^\infty |f_j|^2.
\]

The proof of the following lemma can be found in \cite[Lemma 3.1]{LZZ}.
\begin{lemma}
The following estimate holds:
\[
|f_j|^2 \lesssim \kappa_j^2 \|u(x,\kappa_j)\|^2_{L^2(\Gamma)}
\]
for $j=1, 2, 3, \cdot\cdot\cdot$.
\end{lemma}

\begin{figure}[htbp]
\centering
\includegraphics[width=0.5\textwidth]{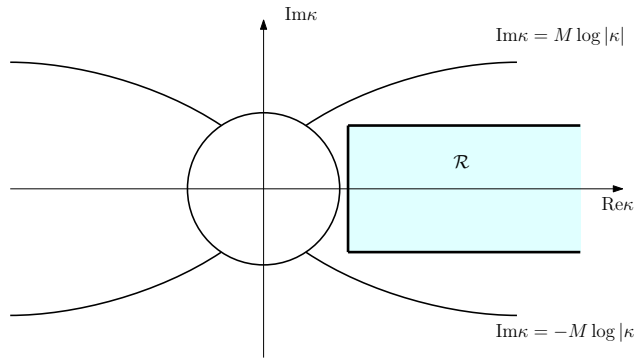}
\caption{The region $\mathcal{R}$.}
\label{dR}
\end{figure}

\begin{lemma}\label{ac_1}
Let $f$ be a real-valued function and $\|f\|_{L^2(\Omega)}\leq Q$.
Then there exist positive constants $d$ and $A, A_1$ satisfying
$C_0< A<A_1$, which do not depend on $f$ and $Q$, such that 
\[
\kappa^2\|u(x,\kappa)\|^2_{L^2(\Gamma)} \lesssim
Q^2e^{c\kappa}\epsilon_1^{2\mu(\kappa)}, \quad \forall\,
\kappa\in (A_1, +\infty),
\]
where $C_0$ is specified in Theorem \ref{bound_2}, $c$ is any
positive constant, and
\[
\epsilon^2_1 :={\rm sup}_{\kappa \in (A,A_1)} \kappa^2\|u\|^2_{L^2(\Gamma)},\quad 
\mu(\kappa) \geq \frac{64ad}{3\pi^2(a^2 + 4d^2)} e^{\frac{\pi}{2d}(\frac{a}{2} -
\kappa)}.
\]
Here $a = A_1 - A.$
\end{lemma}

\begin{proof}
Let 
\[
I(\kappa):=\kappa^2\int_{\Gamma}u(x,\kappa)u(x,-\kappa){\rm d}s,\quad \kappa\in \mathbb{C}.
\]
Since $f(x)$ is a real-valued function,  we have
$\overline{u(x,\kappa)}=u(x,-\kappa)$ for $\kappa\in\mathbb{R}$, which gives
\[
I(\kappa)=\kappa^2\|u(x,\kappa)\|^2_{L^2(\Gamma)},\quad \kappa\in\mathbb{R}.
\]
It follows from Theorem \ref{bound_2} that $I(\kappa)$ is analytic in the domain
\[
\widetilde{\Omega}_M=\{\kappa\in\mathbb{C}:\,-M\log|\kappa|\leq
\mathrm{Im}\,\kappa\leq M\log|\kappa|,\, |\kappa|\geq C_0\},
\]
which is symmetric with respect to the origin. 
Hence, there exists $d>0$ such
that $\mathcal{R} = (A, +\infty)\times (-d, d)\subset\widetilde{\Omega}_M$. The
geometry of domain $\mathcal R$ is shown in Figure \ref{dR}. By
\eqref{bound_3} we have for $\kappa \in \mathcal{R}$ that 
\begin{align*}
|\kappa| \|u(x,\pm\kappa)\|_{L^2(\Gamma)}
& \lesssim|\kappa| \|u(x,\pm\kappa)\|_{H^{1/2}(\Gamma)}\lesssim
|\kappa| \|u(x,\pm\kappa)\|_{H^1(\Omega)}\\
&\lesssim|\kappa |e^{2T({\rm Im}(\pm\kappa))_-}\|f\|_{L^2(\Omega)}
\lesssim |\kappa |e^{2Td}\|f\|_{L^2(\Omega)},
\end{align*}
which shows that 
\[
|\kappa|\|u(x,\pm\kappa)\|_{L^2(\Gamma)} \lesssim|\kappa|\|f\|_{L^2(\Omega)}, \quad\kappa\in \mathcal{R}.
\]
Since 
\[
|I(\kappa)|\leq |\kappa| \|u(x,\kappa)\|_{L^2(\Gamma)} |\kappa|
\|u(x,-\kappa)\|_{L^2(\Gamma)} \lesssim
|\kappa|^2\|f\|_{L^2(\Omega)}^2,\quad\kappa\in \mathcal{R},
\]
we have 
\[
|e^{-c\kappa}I(\kappa)|\lesssim Q^2,\quad\kappa\in \mathcal{R}
\]
for any positive constant $c$. An application of Lemma \ref{ac} leads to 
\[
\big| e^{-c\kappa} I(\kappa)\big| \lesssim
Q^2\epsilon_1^{2\mu(\kappa)},\quad \forall\, \kappa\in (A_1, +\infty),
\]
where 
\[
\mu(\kappa) \geq \frac{64ad}{3\pi^2(a^2 + 4d^2)} e^{\frac{\pi}{2d}(\frac{a}{2} - \kappa)},
\]
which completes the proof. 
\end{proof}

Here we state a simple uniqueness result for the inverse problem. 

\begin{theorem}
Let $f\in L^2(\Omega)$ and $I:=(C_0, C_0 + \delta) \subset \mathbb R^+$ be an
open interval, where $C_0$ is the constant given  in the definition of
$\widetilde{\Omega}_M$ in Lemma \ref{ac_1} and $\delta$ is any positive
constant. Then the source term $f$ can be uniquely determined by the
multi-frequency data $\{u(x,\kappa): x\in \Gamma, \kappa \in I\}\cup
\{u(x,\kappa_j): x\in\Gamma, \kappa_j\in (0, C_0]\}$.
\end{theorem}

\begin{proof}
Let $u(x,\kappa)=0$ for $x\in\Gamma$ and $\kappa \in I \cup \{\kappa_j:
\kappa_j\in (0, C_0])\}$. It suffices to show that $f(x)=0$. Since $u(x,\kappa)$
is analytic in $\widetilde{\Omega}_M$ for $x\in\Gamma$, it holds that
$u(x,\kappa) = 0$ for all eigenvalues $\kappa>C_0$. Then we have that
$u(x,\kappa_j)=0$ for all $\kappa_j, j=1, 2, 3, \cdot\cdot\cdot$.
It follows from \eqref{integral_equation} that 
\[
\int_{\Omega} f(x)\bar{\phi}_j(x) {\rm d}x = 0, \quad j=1, 2, 3, \cdot\cdot\cdot,
\]
which implies $f = 0$.
\end{proof}

The following lemma is important in the stability analysis (cf. \cite[Lemma 3.4]{LZZ}).

\begin{lemma}\label{tail}
Let $f\in H^{n+1}(\Omega)$ and $\|f\|_{H^{n+1}(\Omega)}\leq Q$. It holds that 
\[
\sum_{j \geq s} |f_j|^2 \lesssim \frac{Q^2}{s^{\frac{2}{3}(n+1)}}.
\]
\end{lemma}

Define a real-valued function space
\[
\mathcal C_Q = \{f \in H^{n+1}(\Omega):  \|f\|_{H^{n+1}(\Omega)}\leq Q, ~ \text{supp}
f\subset \Omega, ~ f: \Omega \rightarrow \mathbb R \}.
\]
Now we are in the position to discuss the inverse source problem. Let $f\in
\mathcal C_Q$. The inverse source problem is to determine $f$ from the limited aperture boundary
data $u(x,\kappa)$, $x\in\Gamma$, $\kappa \in (A, A_1) \cup \cup_{j=1}^{N_1}
\kappa_j$, where $1\leq N_1 \in \mathbb N$ and $\kappa_{N_1}>A_1$. Here $A$ and $A_1$
are the constants specified in Lemma \ref{ac_1}.

The following stability estimate is the main result of this paper.

\begin{theorem}
Let $u(x,\kappa)$ be the solution of the scattering problem corresponding to the source $f\in \mathcal C_Q$.
Then for sufficiently small $\epsilon_1$, the following estimate holds: 
\begin{align}\label{stability}
\|f\|_{L^2(\Omega)}^2  \lesssim \epsilon^2 +
\frac{Q^2}{N_1^{\frac{1}{3}(n+1)}(\ln|\ln\epsilon_1|)^{\frac{1}{3}(n+1)}},
\end{align}
where
\begin{align*}
\epsilon^2 = \sum_{j=1}^{N_1} \kappa_j^2 \|u(x,\kappa_j)\|^2_{L^2(\Gamma)},\quad
\epsilon^2_1 = {\rm sup}_{\kappa \in (A,A_1)} \kappa^2 \|u(x,\kappa)\|^2_{L^2(\Gamma)}.
\end{align*}
\end{theorem}

\begin{proof}
We can assume that $\epsilon_1 \leq e^{-1}$, otherwise the estimate is obvious.

First, we link the data $\kappa^2 \|u(x,\kappa)\|^2_{L^2(\Gamma)}$
for large wavenumber $\kappa$ satisfying $\kappa\leq L$ with the given data
$\epsilon_1$ of small wavenumber by using the analytic continuation in Lemma
\ref{ac_1}, where $L$ is some large positive integer to be determined later.  By
Lemma \ref{ac_1}, we obtain 
\begin{align*}
&\kappa^2\|u(x,\kappa)\|^2_{L^2(\Gamma)}\lesssim Q^2e^{c|\kappa|} \epsilon_1^{\mu(\kappa)}\\
&\lesssim Q^2{\rm exp}\{c\kappa - \frac{c_2a}{a^2 + c_3}e^{c_1(\frac{a}{2} - \kappa)}
|{\ln}\epsilon_1|\}\\
&\lesssim  Q^2{\rm exp} \{  - \frac{c_2a}{a^2 + c_3}e^{c_1(\frac{a}{2} - \kappa)}|{\ln}\epsilon_1| (1 -  \frac{c_4\kappa(a^2 + c_3)}{a} e^{c_1(\kappa - \frac{a}{2})}|{\ln}\epsilon_1|^{-1})\}\\
&\lesssim Q^2{\rm exp} \{  - \frac{c_2a}{a^2 + c_3}e^{c_1(\frac{a}{2} - L)}|{\ln}\epsilon_1| (1 -  \frac{c_4L(a^2 + c_3)}{a} e^{c_1(L - \frac{a}{2})}|{\ln}\epsilon_1|^{-1})\}\\
&\lesssim Q^2{\rm exp} \{  - b_0e^{-c_1L}|{\ln}\epsilon_1| (1 - b_1L e^{c_1L }|{\ln}\epsilon_1|^{-1})\},
\end{align*}
where $c, c_i, i=1,2$ and $b_0, b_1$ are constants. Let
\begin{align*}
L = 
\begin{cases}
\left[\frac{1}{2c_1}\ln|\ln \epsilon_1|\right], &\quad N_1\leq \frac{1}{2c_1} \ln|\ln\epsilon_1|,\\
N_1, &\quad N_1>  \frac{1}{2c_1}\ln|\ln\epsilon_1|.
\end{cases}
\end{align*}

If $N_1\leq  \frac{1}{2c_1}\ln|\ln\epsilon_1|$, we obtain for $\epsilon_1$
sufficiently small that 
\begin{align*}
\kappa^2\|u(x,\kappa)\|^2_{L^2(\Gamma)}&\lesssim Q^2{\rm exp} \{  - b_0e^{-c_1L}|{\ln}\epsilon_1| (1 - b_1L e^{c_1L }|{\ln}\epsilon_1|^{-1})\}\\
& \lesssim Q^2\exp\{-\frac{1}{2}b_0e^{-c_1L}|\ln \epsilon_1|\}.
\end{align*}
Noting $e^{-x}\leq \frac{(2n+3)!}{x^{2n+3}}$ for $x>0$, we obtain
\begin{align*}
\sum_{j=N_1+1}^L\kappa_j^2\|u(x,\kappa_j)\|^2_{L^2(\Gamma)} \lesssim Q^2
Le^{(2n+3)c_1L}|\ln\epsilon_1|^{-(2n+3)}.
\end{align*}
Taking $L=\frac{1}{2c_1}\ln|\ln\epsilon_1|$, combining the above estimates
and  Lemma \ref{tail}, we get
\begin{align*}
&\|f\|_{L^2( \Omega)}^2\lesssim \sum_{j=1}^{N_1} |f_j|^2+\sum_{j=N_1+1}^L |f_j|^2 + \sum_{j=L+1}^{+\infty} |f_j|^2\\
&\lesssim \sum_{j=1}^{N_1} \kappa_j^2\|u(x,\kappa_j)\|^2_{L^2(\Gamma)}+\sum_{j=N_1+1}^L \kappa_j^2\|u(x,\kappa_j)\|^2_{L^2(\Gamma)}+ \frac{1}{L^{\frac{2}{3}(n+1)}}\|f\|^2_{H^{n+1}(\Omega)}\\
&\lesssim \epsilon^2 + LQ^2e^{(2n+3)c_1L}|\ln\epsilon_1|^{-(2n+3)}+ \frac{Q^2}{L^{\frac{2}{3}(n+1)}}\\
& \lesssim \epsilon^2 + Q^2\left((\ln|\ln\epsilon_1|)|\ln\epsilon_1|^{\frac{2n+3}{2}}|\ln\epsilon_1|^{-(2n+3)}+(\ln|\ln\epsilon_1|)^{-\frac{2}{3}(n+1)}\right)\\
& \lesssim \epsilon^2 + Q^2\left((\ln|\ln\epsilon_1|)|\ln\epsilon_1|^{-\frac{2n+3}{2}}+(\ln|\ln\epsilon_1|)^{-\frac{2}{3}(n+1)}\right)\\
& \lesssim \epsilon^2 + Q^2(\ln|\ln\epsilon_1|)^{-\frac{2}{3}(n+1)}\\
&\lesssim \epsilon^2 +
\frac{Q^2}{N_1^{\frac{1}{3}(n+1)}(\ln|\ln\epsilon_1|)^{\frac{1}{3}(n+1)}},
\end{align*}
where we have used $|\ln\epsilon_1|^{1/2}\geq \ln|\ln\epsilon_1|$ for
sufficiently small $\epsilon_1$.

If $N_1 >  \frac{1}{2c_1}\ln|\ln\epsilon_1|$, we have from  Lemma \ref{tail} that 
\begin{align*}
\|f\|_{L^2( \Omega)}^2 &\lesssim \sum_{j=1}^{N_1} |f_j|^2+ \sum_{j=N_1+1}^{+\infty} |f_j|^2\\
&\lesssim \epsilon^2 + \frac{Q^2}{N_1^{\frac{2}{3}(n+1)}}\\
&\lesssim \epsilon^2 +
\frac{Q^2}{N_1^{\frac{1}{3}(n+1)}(\ln|\ln\epsilon_1|)^{\frac{1}{3}(n+1)}},
\end{align*}
which completes the proof.
\end{proof}

\begin{remark}
The stability \eqref{stability} consists of two parts: the data discrepancy and
the high frequency tail. The former is of the Lipschitz type. The latter
decreases as $N_1$ increases which makes the problem have an almost Lipschitz
stability. The result reveals that the problem becomes more stable when higher
frequency data is used.
\end{remark}

\begin{remark}
Following the above arguments we can also prove an increasing stability for the inverse source problem in a two-dimensional waveguide. 
\end{remark}

\section{Conclusion}\label{con}

In this paper, we have studied the resonance-free region and given an upper bound for the resolvent of the Schr\"{o}dinger operator in a planar waveguide in three dimensions. As an application of the resolvent analysis, the direct source problem is shown to have a unique solution under some appropriate assumptions. Moreover, an increasing stability result is presented for the inverse source problem. The stability analysis requires the limited aperture Dirichlet data only at multiple frequencies. The estimate consists  of the data discrepancy and the high frequency tail of the source function. It shows that the ill-posedness of the inverse source problem decreases as the frequency increases for the data. We believe that the method can be directly applied to the planar waveguide with different boundary conditions, such as the Neumann boundary condition, and tubular waveguides. A related but more challenging problem is to study the inverse problem of determining the potential. The progress will be reported elsewhere in the future.  

\appendix

\section{Two useful lemmas}

The following lemma  gives a link between the values of an analytical function for small and large arguments.
The proof can be found in \cite[Lemma A.1]{LZZ}.
\begin{lemma}\label{ac}
Let $p(z)$ be analytic in the infinite rectangular slab
\[
R = \{z\in \mathbb C: (A, +\infty)\times (-d, d) \}, 
\]
where $A$ is a positive constant, and continuous in $\overline{R}$ satisfying
\begin{align*}
\begin{cases}
|p(z)|\leq \epsilon, &\quad z\in (A, A_1],\\
|p(z)|\leq M, &\quad z\in R,
\end{cases}
\end{align*}
where $A, A_1, \epsilon$ and $M$ are positive constants. Then there exists a function $\mu(z)$ with $z\in (A_1, +\infty)$ satisfying 
\begin{equation*}
\mu(z) \geq \frac{64ad}{3\pi^2(a^2 + 4d^2)} e^{\frac{\pi}{2d}(\frac{a}{2} - z)},
\end{equation*}
where $a = A_1 - A$, such that
\begin{align*}
|p(z)|\leq M\epsilon^{\mu(z)}\quad \forall\, z\in (A_1, +\infty).
\end{align*}
\end{lemma}

\begin{lemma}\label{eigenfunction_est1}
The following estimate holds in $\mathbb R^n$:
\begin{align}\label{boundary_estimate}
\|\partial_\nu \phi_j\|_{L^2(\partial \Omega)}\leq C(n)\kappa_j,\quad j\in\mathbb N,
\end{align}
where the positive constant $C(n)$ is independent of $j$ and dependent on the dimension $n$. Moreover, the Dirichlet eigenvalues $\{\mu_j\}_{j=1}^\infty$ satisfy the Weyl-type inequality
\begin{align}\label{weyl}
E_1 j^{2/n}\leq \mu_j\leq E_2 j^{2/n},
\end{align}
where $E_1$ and $E_2$ are two positive constants independent of $j$.
\end{lemma}

\begin{proof}
We begin with the estimate \eqref{boundary_estimate} for the eigenfunctions on the boundary.
Let $u$ be a Dirichlet eigenfunction of $P=-\Delta + V$ in $\Omega$. For any
differential operator $A$, by \cite[Lemma 2.1]{hassell2002upper}, we have the
Rellich-type identity 
\begin{equation}\label{rellich}
\int_{\Omega}u[P,A]u{\rm d}x=\int_{\partial \Omega}\partial_\nu
u Au{\rm d}s,
\end{equation}
where $[P, A]=PA-AP$. In fact, let $\lambda$ be the eigenvalue corresponding to
the eigenfunction $u$. We have $[P,A]=[P-\lambda,A]$ and $(P-\lambda)u=0$. A
simple calculation yields
\begin{align*}
&\int_{\Omega} u[P,A]u{\rm d}x\\
&=\int_{\Omega}\left[u(-\Delta Au + VAu)-\lambda  uAu\right]{\rm d}x\\
&=\int_{\Omega}\left[-\Delta u Au + VuAu - \lambda uAu\right]{\rm d}x+\int_{\partial
\Omega}\partial_\nu u Au {\rm d}s\\
&=\int_{\Omega}(P-\lambda)u Au{\rm d}x+\int_{\partial \Omega}\partial_\nu
u Au{\rm d}s\\
&=\int_{\partial \Omega}\partial_\nu u Au{\rm d}s.
\end{align*}

Now let $u$ be a normalized Dirichlet eigenfunction with eigenvalue $\lambda$. 
Since $\Omega$ is a convex domain, we take the normal vector $\nu(x)$ on $\partial\Omega$ as a function defined on the unit sphere
$\mathbb S^{n - 1}$ such that $\nu(x) = \nu (\hat{x})$.  Then we define a function $\tilde{\nu}(x)$ such that $\tilde{\nu}(x) = \nu(\hat{x})$ for $x\in\Omega$.
Choose $A = \tilde{\nu}(x)\cdot\nabla$. It is clear that the right hand side of
\eqref{rellich} is exactly $\|\partial_\nu u\|^2_{L^2(\partial \Omega)}$. It follows
from the integration by parts that the left hand side of \eqref{rellich} can be
written as
\[
\int_{\Omega}(B_1u)(B_2u){\rm d}x,
\]
where $B_1$ and $B_2$ are two first order differential operators. The lemma is
proved by observing that 
 \[
 \int_{\Omega}(B_1u)(B_2u){\rm d}x\leq C\int_{\Omega}\nabla u\cdot \nabla
u\,\mathrm{d}x\leq C\lambda,
 \]
where the positive constant $C$ does not depend on $\lambda$.

Next, we prove the Weyl-type inequality \eqref{weyl}.
Let $\{\mu_j, \phi_j\}_{j=1}^\infty$ be the positive increasing eigenvalues and eigenfunctions of $-\Delta + V$ in $\Omega$.
Then we have following min-max principle: 
\[
\mu_j=\sup_{\phi_1,\cdots,\phi_{j-1}}\inf_{\psi\in[\phi_1,\cdots,
\phi_{j-1}]^\perp\atop \psi\in H_0^1(\Omega)}\frac{\int_{\Omega}|\nabla
\psi|^2 + V|\psi|^2{\rm d}x}{\int_{\Omega}\psi^2{\rm d}x}.
\]
Assume $C_1<\|V(x)\|_{L^\infty(\Omega)}<C_2$ on $\Omega$, where $C_1,C_2$ are two constants. Assume $
\mu_1^{(k)}\leq\mu_2^{(k)}\leq\cdots$ are the eigenvalues for the operator
$-\Delta + C_k$ for $k=1,2$. By the min-max principle, we have
\[
\mu_j^{(2)}<\mu_j<\mu_j^{(1)}, \quad j=1, 2, \dots.
\]
We have from Weyl's law \cite{Weyl} for $-\Delta + C_k$ that 
\[
\lim_{j\rightarrow+\infty}\frac{\mu_j^{(k)}}{j^{2/n}}=D_k,
\]
where $D_k$ is a constant. Therefore there exist two constants $E_1$ and $E_2$
such that 
\[
E_1 j^{2/n}\leq \mu_j\leq E_2 j^{2/n},
\]
which completes the proof.
\end{proof}


\begin{thebibliography}{99}

\bibitem{AAB}
S. Acosta, R. Alonso, and L. Borcea, Source estimation with incoherent waves in random waveguides, Inverse Probl., 31 (2015) 035013.

\bibitem{AGA}
T. Arens, D. Gintides, and A. Lechleiter, Direct and inverse medium scattering in a three-dimensional homogeneous planar waveguide, SIAM J. Appl. Math., 71 (2011), 753--772.

\bibitem{BL}
G. Bao and P. Li, Maxwell's Equations in Periodic Structures, Series on Applied Mathematical Sciences, Springer, to appear. 

\bibitem{blt}
G. Bao, J. Lin, and F. Triki, A multi-frequency inverse source problem, J.
Differential Equations, 249 (2010), 3443--3465.

\bibitem{blz}
G. Bao, P. Li, and Y. Zhao, Stability for the inverse source problems in elastic
and electromagnetic waves, J. Math. Pures Appl., 134 (2020), 122--178.

\bibitem{BT}
G. Bao and F. Triki, Stability for the multifrequency inverse medium problem, J. Differential
Equations, 269 (2020), 7106--7128.

\bibitem{bc-77}
N. Bleistein and J. K. Cohen, Nonuniqueness in the inverse source problem in
acoustics and electromagnetics, J. Math. Phys., 18 (1977), 194--201.

\bibitem{BCM}
L. Borcea, F. Cakoni, and S. Meng, A direct approach to imaging in a waveguide with perturbed geometry, J. Comput. Phys., 392 (2019) 556--577.
 
\bibitem{BIT}
L. Borcea, L. Issa, and C. Tsogka, Source localization in random acoustic waveguides, Multiscale Model. Simul., 8 (2010), 
1981--2022. 

\bibitem{BGWX}
J. L. Buchanan, R. P. Gilbert, A. Wirgin, and Y. Xu, Marine Acoustics: Direct and
Inverse Problems, SIAM, Philadelphia, 2004.

\bibitem{CCH}
F. Cakoni, D. Colton, and H. Haddar, A duality between scattering poles and transmission
eigenvalues in scattering theory, Proc. R. Soc. A, 476 (2020), 20200612.

\bibitem{cheng2016increasing}
J. Cheng, V. Isakov, and S. Lu, Increasing stability in the inverse source
problem with many frequencies, J. Differential Equations, 260 (2016),
4786--4804.

\bibitem{collin}
R. E. Collin, Field Theory of Guided Waves, vol. 2, IEEE Press, New York, 1991.

\bibitem{CK}
D. Colton and R. Kress, Inverse Acoustic and Electromagnetic Scattering Theory, Applied Mathematical Sciences, vol. 93, Springer-Verlag, Berlin, 1998.

\bibitem{ds-82}
A. Devaney and G. Sherman, Nonuniqueness in inverse source and scattering
problems, IEEE Trans. Antennas Propag., 30 (1982), 1034--1037.

\bibitem{Dyatlov}
S. Dyatlov and M. Zworski, Mathematical Theory of Scattering Resonances, vol. 200, American
Mathematical Soc., 2019.

\bibitem{ei-18}
M. Entekhabi and V. Isakov, On increasing stability in the two dimensional
inverse source scattering problem with many frequencies, Inverse Problems, 34
(2018), 055005.

\bibitem{ei-19}
M. Entekhabi and V. Isakov, Increasing stability in acoustic and elastic inverse
source problems, SIAM J. Math. Anal., 52 (2020), 5232--5256. 

\bibitem{FY06}
D.  Finco and K. Yajima, The $L^p$ boundedness of wave operators for Schr\"{o}dinger operators with threshold singularities, II. Even dimensional case, J. Math. Sci. Univ. Tokyo, 13 (2006), 277--346. 


\bibitem{hassell2002upper}
A. Hassell and T. Tao, Upper and lower bounds for normal derivatives of Dirichlet eigenfunctions, arXiv:math/0202140.  

\bibitem{KLU}
K. Krupchyh, M. Lassas, and G. Uhlmann, Inverse problems with partial data for a magnetic Schr\"{o}dinger operator in an infinite slab and on a bounded domain, Comm. Math. Phys., 312 (2012), 87--126.

\bibitem{LYZ}
P. Li, X. Yao, and Y. Zhao, Stability for an inverse source problem of the biharmonic operator, arXiv:2102.04631.

\bibitem{li2017increasing}
P. Li and G. Yuan, Increasing stability for the inverse source scattering
problem with multi-frequencies, Inverse Problems and Imaging, 11 (2017),
745--759. 

\bibitem{LZZ}
P. Li, J. Zhai, and Y. Zhao, Stability for the acoustic inverse source problem in inhomogeneous media, SIAM J. Appl. Math., 80 (2020), 2547--2559.

\bibitem{LZ}
P. Li and Y. Zhao, Stability for the electromagnetic inverse source problem in
inhomogeneous medium, preprint.

\bibitem{LU}
X. Li and G.  Uhlmann, Inverse problems with partial data in a slab, Inverse Probl. Imaging, 4 (2010), 449--462.

\bibitem{liu}
K. Liu, Fast imaging of sources and scatterers in a stratified ocean waveguide, SIAM J. Imaging Sci., 14 (2021), 224--245.

\bibitem{LXZ}
K. Liu, Y. Xu, and J. Zou, A multilevel sampling method for detecting sources in a stratified ocean waveguide, J. Comput. Appl. Math., 309 (2017), 95--110. 

\bibitem{marcuse}
D. Marcuse, Theory of Dielectric Optical Waveguides, Academic Press, New York, 1974.

\bibitem{MS}
P. Monk and V. Selgas, Sampling type methods for an inverse waveguide problem, Inverse Probl. Imaging, 6 (2012), 709--747.

\bibitem{Weyl}
H. Weyl, Das asymptotische verteilungsgesetz der eigenwerte linear partieller differentialgleichungen, Math. Ann., 71 (1911), 441--479.

\end{thebibliography}
\end{document}